\documentclass{amsart}

\usepackage{amsmath}
\usepackage{amssymb}
\usepackage{indentfirst}
\usepackage{galois}
\usepackage{amsthm}
\usepackage{amsfonts}

\newtheorem{theorem}{Theorem}[section]
\newtheorem{proposition}[theorem]{Proposition}
\newtheorem{lemma}[theorem]{Lemma}
\newtheorem{corollary}[theorem]{Corollary}
\newtheorem{claim}[theorem]{Claim}

\newtheorem{question}[theorem]{Question}

\newtheorem*{main}{Main Result}

\theoremstyle{definition}
\newtheorem{definition}[theorem]{Definition}

\theoremstyle{remark}
\newtheorem{remark}[theorem]{Remark}

\numberwithin{equation}{section}

\begin{document}

\bibliographystyle{plain}
\title{Complex symmetric composition operators with automorphic symbols}

\author[Y.X. Gao and Z.H. Zhou] {Yong-Xin Gao and Ze-Hua Zhou$^*$}
\address{\newline  Yongxin Gao\newline Department of Mathematics,
Tianjin University, Tianjin 300350, P.R. China.}

\email{tqlgao@163.com}

\address{\newline Zehua Zhou\newline Department of Mathematics, Tianjin University, Tianjin 300350,
P.R. China.}
\email{zehuazhoumath@aliyun.com; zhzhou@tju.edu.cn}


\date{}
\thanks{\noindent $^*$Corresponding author.\\
This work was supported in part by the National Natural Science
Foundation of China (Grant Nos. 11371276; 11301373; 11401426).}

\begin{abstract}
In this paper we show that a composition operator $C_\varphi$ cannot be complex symmetric on Hardy-Hilbert space $H^2(D)$ when $\varphi$ is an elliptic automorphism of order $3$ and not a rotation. This completes the project of finding all invertible composition operators which are complex symmetric $H^2(D)$.
\end{abstract}

\maketitle

\section{Introduction}

Let $T$ be a bounded operator on a complex Hilbert space $\mathcal{H}$. Then $T$ is called complex symmetric if there exists a conjugation $C$ such that $T=CT^*C$. Here a conjugation is a conjugate-linear, isometric involution on $\mathcal{H}$. The operator $T$ may also be called $C$-symmetric if $T$ is complex symmetric with respect to a specifical conjugation $C$. For more details about complex symmetric operators one may turn to \cite{CS1} and \cite{Cs2}.

In this paper, we are particularly interested in the complex symmetry of composition operators induced by analytic self-maps of $D$. This subject was started by Garcia and Hammond in \cite{GH}.

Recall that for each analytic self-map $\varphi$ of the unit disk $D$, the composition operator given by
$$C_\varphi f=f\comp\varphi$$
is always bounded on $H^2(D)$. Here, the Hardy-Hilbert space $H^2(D)$ is the set of analytic functions on $D$ such that
$$||f||_{H^2}^2=\sup_{0<r<1}\int_{0}^{2\pi}|f(re^{i\theta})|^2\frac{d\theta}{2\pi}<\infty.$$

Several simple examples of complex symmetric composition operators on $H^2(D)$ arise immediately. For example, every normal operator is complex symmetric (see \cite{CS1}), so when $\varphi(z)=sz$ with $|s|\leqslant 1$, $C_\varphi$ is normal hence complex symmetric on $H^2(D)$. Also, Theorem 2 in \cite{Cs3} states that
each operator satisfying a polynomial equation of order $2$ is complex symmetric. So when $\varphi$ is an elliptic automorphism of order $2$, then $C_\varphi^2=I$, thus $C_\varphi$ is complex symmetric on $H^2(D)$. In \cite{Invo} Waleed Noor find the conjugation $C$ such that $C_\varphi$ is $C$-symmetric when $\varphi$ is an elliptic automorphism of order $2$.

However, we are still far away from our final destination: finding out all composition operators which are complex symmetric on $H^2(D)$. The first step is, of course, to determine the complex symmetric composition operators induced by automorphisms. Bourdon and Waleed Noor considered this problem in \cite{BN}. The next Proposition is Proposition 2.1 in \cite{BN}

\begin{proposition}\label{propa}
Let $\varphi$ be a self-map of $D$. If $C_\varphi$ is complex symmetric on $H^2(D)$, then $\varphi$ has a fixed point in $D$.

Particularly, if $\varphi$ is either a parabolic or a hyperbolic automorphism, then $C_\varphi$ cannot be complex symmetric on $H^2(D)$.
\end{proposition}

Thanks to this proposition, we only need to investigate the elliptic automorphisms of the unit disk $D$. It turns out that things depend much on the orders of the automorphisms. The next proposition is one of the main result in \cite{BN}.

\begin{proposition}\label{propb}
Let $\varphi$ be an elliptic automorphism of order $q$. If $q=2$, then $C_\varphi$ is always complex symmetric on $H^2(D)$. If $4\leqslant q\leqslant\infty$, then $C_\varphi$ is complex symmetric on $H^2(D)$ only if $\varphi$ is a rotation.
\end{proposition}

However, the order 3 elliptic case remains as an open question, which is posed by Bourdon and Waleed Noor in \cite{BN}:

\begin{question}
Is $C_\varphi$ complex symmetric on $H^2(D)$ when $\varphi$ is an elliptic automorphism of order $3$?
\end{question}

The aim of this paper is to solve this problem. By dong this we complete the project of finding out all invertible composition operators which are complex symmetric $H^2(D)$.

In our main result Theorem \ref{main}, we prove that if $\varphi$ is an elliptic automorphism of order $3$ and not a rotation, then the composition operator $C_\varphi$ cannot be complex symmetric on $H^2(D)$. Then we can come to our final result as follows. It will be given as Corollary \ref{last} in the third section.

\begin{main}\label{fin}
Suppose $\varphi$ is an automorphism of $D$. Then $C_\varphi$ is complex symmetric on $H^2(D)$ if and only if $\varphi$ is either a rotation or an elliptic automorphism of order two.
\end{main}

\section{Preliminary}

The Hardy space $H^2(D)$ is naturally a Hilbert space, with the inner product
$$\langle f,g\rangle=\sup_{0<r<1}\int_{0}^{2\pi}f(re^{i\theta})\overline{g(re^{i\theta})}\frac{d\theta}{2\pi}.$$
For each $w\in D$, let
$$K_w(z)=\frac{1}{1-\overline{w}z}.$$
Then $K_w\in H^2(D)$ is the reproducing kernel at the point $w$, i.e., $$\langle f,K_w\rangle=f(w)$$ for all $f\in H^2(D)$.

It is well known that the automorphisms of the unit disk $D$ fall into three categories: parabolic and hyperbolic automorphisms have not fixed point in $D$, and besides them are the elliptic automorphisms who have a unique fixed point in $D$.

\begin{definition}
The order of an elliptic automorphism $\varphi$ is the smallest positive integer such that $\varphi_n(z)=z$ for all $z\in D$. Here $\varphi_n=\varphi\comp\varphi\comp...\comp\varphi$ denotes the $n$-th iterate of $\varphi$. If no such positive integer exists, then $\varphi$ is said to have order $\infty$.
\end{definition}

Note that if the order of an automorphism $\varphi$ is one, then $\varphi$ is identity on $D$. If $\varphi$ has order two, then $\varphi$ is of the form
$$\varphi(z)=\varphi_a(z)=\frac{a-z}{1-\overline{a}z}$$
for some $a\in D$.
\begin{remark}
The notation $\varphi_a$ will be used throughout this paper. $\varphi_a$ is the involution automorphism exchanges $0$ and $a$.
\end{remark}

By Propositions \ref{propa} and \ref{propb}, we only need to be concerned with the elliptic automorphisms that have order $3$. Moreover, if the fixed point of a automorphism is $0$, then it is a rotation and of course complex symmetric. So throughout this paper we will always assume that $\varphi$ is an elliptic automorphism of order $3$ with fixed point $a\in D\backslash\{0\}$.
\\

For a complex symmetric operator $T$, one should keep the following simple result in mind.

\begin{lemma}
Suppose $T$ is $C$-symmetric on $\mathcal{H}$. Then $\lambda\in\mathbb{C}$ is a eigenvalue of $T$ if and only if $\overline{\lambda}$ is a eigenvalue of $T^*$. Moreover, the conjugation $C$ maps the eigenvectors subspace $Ker(T-\lambda)$ onto $Ker(T^*-\overline{\lambda})$.
\end{lemma}

\begin{proof}
One only need to note that $T=CT^*C$ implies  $T-\lambda=C(T^*-\overline{\lambda})C$.
\end{proof}

The next lemma gives the eigenvectors of $C_\varphi$ on $H^2(D)$ when $\varphi$ is an elliptic automorphism of order $3$. The main calculation of this lemma is done in \cite{BN}, with the help of Theorem 9.2 in Cowen and MacCluer's book \cite{CC}.

\begin{lemma}\label{le}
Suppose $\varphi$ is an elliptic automorphism of order $3$ with fixed point $a\in D$. Let $\Lambda_m=Ker(C_\varphi-\varphi'(a)^m)$ and $\Lambda^*_m=Ker(C_\varphi^*-\overline{\varphi'(a)}^m)$ for $m=0,1,2$, then
$$\Lambda_m=\overline{ \mathrm{span}}\{\varphi_a^{3j+m}; j=0,1,2,...\}$$ and $$\Lambda_m^*=\overline{\mathrm{span}}\{e_{3j+m}-ae_{3j+m-1}; j=0,1,2,...\},$$
where $e_{-1}=0$ and $e_k=K_a\varphi_a^k$ for $k=0,1 ,2,...$
\end{lemma}

\begin{proof}
Let $\tau=\varphi_a\comp\varphi\comp\varphi_a$. Then $\tau$ is a rotation of order $3$, that is, $\tau(z)=\varphi'(a)z$. So $C_\tau z^k=\varphi'(a)^kz^k$ for $k=0,1,2,...$ Moreover, since $C_\tau^*=C_{\tau^{-1}}$ where $\tau^{-1}(z)=\overline{\varphi'(a)}z$, we also have $C_\tau^* z^k=\overline{\varphi'(a)}^kz^k$. Thus
$$z^k\in Ker(C_\tau-\varphi'(a)^k)\cap Ker(C_\tau^*-\overline{\varphi'(a)}^k)$$
for $k=0,1,2,...$

Now by the definition of $\tau$ we have $C_\varphi C_{\varphi_a}=C_\tau C_{\varphi_a}$ and $C_\varphi^* C_{\varphi_a}^*=C_\tau^* C_{\varphi_a}^*$, so
$C_{\varphi_a}z^k\in Ker(C_\varphi-\varphi'(a)^k)$ and $C_{\varphi_a}^*z^k\in Ker(C_\varphi^*-\overline{\varphi'(a)}^k)$. Since $\varphi'(a)^3=1$ one can get
$$\overline{span}\{C_{\varphi_a}z^{3j+m}; j=0,1,2,...\}\subset\Lambda_m$$
and
$$\overline{span}\{C_{\varphi_a}^*z^{3j+m}; j=0,1,2,...\}\subset\Lambda_m^*.$$
On the other hand, both $C_{\varphi_a}$ and $C_{\varphi_a}^*$ are invertible on $H^2(D)$, then
$$\overline{span}\{C_{\varphi_a}z^k; k=0,1,2,...\}=\overline{span}\{C_{\varphi_a}^*z^k; k=0,1,2,...\}=H^2(D).$$
so we have
$$\Lambda_m=\overline{span}\{C_{\varphi_a}z^{3j+m}; j=0,1,2,...\}$$
and
$$\Lambda_m^*=\overline{span}\{C_{\varphi_a}^*z^{3j+m}; j=0,1,2,...\}.$$

Finally, $C_{\varphi_a}z^k=\varphi_a^k$ and the proof of Lemma 2.2 in \cite{BN} shows that $C_{\varphi_a}z^k=e_k-ae_{k-1}$, hence the proof is done.
\end{proof}

\begin{remark}
Note that $||e_j||^2=(1-|a|^2)^{-1}$ for $j=0,1,2,...$ and $\langle e_j,e_k\rangle=0$ whenever $j\ne k$.
\end{remark}

\begin{remark}\label{re2}
If $C_\varphi$ is $C$-symmetric with respect to a conjugation $C$, then $C\Lambda_m=\Lambda_m^*$ for $m=0,1,2$.
\end{remark}

\section{Proof of the main result}

In this section, we will focus on the proof of our main result, i.e., Theorem \ref{main}, which assert that no elliptic automorphism of order $3$ except for rotations can induce a complex symmetric composition operator on $H^2(D)$.

We would like to point out here that throughout the rest of this paper, each notation will always represent the same thing as it
did initially. For example, $\varphi$ is always a elliptic automorphism of order $3$ in what follows, $a$ is always the fixed point of $\varphi$ in $D\backslash\{0\}$, and $\rho$ always represents the same constant $-\frac{\overline{a}^2}{a}\cdot\frac{1-|a|^2}{1-|a|^4}$.

We will assume that $C_\varphi$ is $C$-symmetric with respect to some conjugation $C$, and finally we will show this assumption leads to a contradiction. Now we start by determining the image of a certain vector under the conjugation $C$. The notation $\{e_j\}_{j=0}^\infty$ in Lemma \ref{le} is still valid in this section.

\begin{claim}\label{1}
Let $\varphi$ be an elliptic automorphism of order $3$ with fixed point $a\in D\backslash\{0\}$. If $C_\varphi$ if $C$-symmetric on $H^2(D)$ with respect to a conjugation $C$, then we have
$$Ce_0=c_0\frac{1-\overline{a}^3\varphi_a^3}{1-\rho\varphi_a^3},$$
where $c_0$ is a constant and $\rho=-\frac{\overline{a}^2}{a}\cdot\frac{1-|a|^2}{1-|a|^4}$.
\end{claim}

\begin{proof}
Let $h_0=Ce_0$. Since $e_0\in\Lambda_0^*$, we have $h_0\in\Lambda_0$. So we can suppose that
$$h_0=\sum_{j=0}^\infty c_j\varphi_a^{3j}.$$

It is obvious that $e_0$ is orthogonal to $\Lambda_2^*$. So by Remark \ref{re2} and the fact that $C$ is an isometry, $h_0$ is orthogonal to $\Lambda_2$, which means that $\langle h_0,\varphi_a^{3k+2}\rangle=0$ for $k=0,1,2,...$ So we have the following equations,
\begin{align}\label{3.1}
\sum_{j=0}^{k}c_j\overline{a}^{3k+2-3j}+\sum_{j=k+1}^{\infty}c_ja^{3j-3k-2}=0
\end{align}
for $k=0,1,2...$ Replace $k$ by $k+1$ in (\ref{3.1}) we get
$$\sum_{j=0}^{k+1}c_j\overline{a}^{3k+5-3j}+\sum_{j=k+2}^{\infty}c_ja^{3j-3k-5}=0,$$
hence
\begin{align}\label{3.2}
\sum_{j=0}^{k+1}c_j\overline{a}^{3k+5-3j}a^3+\sum_{j=k+2}^{\infty}c_ja^{3j-3k-2}=0.
\end{align}

Combining (\ref{3.1}) and (\ref{3.2}), we have
\begin{align}\label{3.3}
\sum_{j=0}^kc_j\overline{a}^{3k+2-3j}+c_{k+1}a\frac{1-|a|^4}{1-|a|^6}=0
\end{align}
for $k=0,1,2...$
Thus, we have $c_1=\tilde\rho c_0$ and $c_{j+1}=\rho c_j$ for $j=1,2,3...$, where
$\tilde\rho=-\frac{\overline{a}^2}{a}\cdot\frac{1-|a|^6}{1-|a|^4}$ and $\rho=-\frac{\overline{a}^2}{a}\cdot\frac{1-|a|^2}{1-|a|^4}$.
Therefore,
\begin{align*}
h_0&=c_0+c_1\sum_{j=1}^\infty\rho^{j-1}\varphi_a^{3j}
\\&=c_0+c_0\frac{\tilde\rho\varphi_a^3}{1-\rho\varphi_a^3}
\\&=c_0\frac{1-\overline{a}^3\varphi_a^3}{1-\rho\varphi_a^3}.
\end{align*}
\end{proof}

\begin{claim}
For the constant $c_0$ in Claim \ref{1}, we have
$$|c_0|=\frac{1}{1-|a|^4}.$$
\end{claim}

\begin{proof}
Let
$$g=\frac{\overline{\rho}-\varphi_a^3}{1-\rho\varphi_a^3},$$
then $g$ is an inner function and $g(0)=-\frac{a^2}{\overline{a}}$.

A easy calculation shows that $h_0=\gamma_1g+\gamma_2$, where
$$\gamma_1=c_0\frac{\overline{a}^2}{a}(1+|a|^2) , \gamma_2=c_0(1+|a|^2).$$
So we have
\begin{align*}
||h_0||^2&=\langle\gamma_1g+\gamma_2,\gamma_1g+\gamma_2\rangle
\\&=|\gamma_1|^2+|\gamma_2|^2+2\Re\{\gamma_1\overline{\gamma_2}g(0)\}
\\&=|c_0|^2(1-|a|^2)(1+|a|^2)^2.
\end{align*}

Since $C$ is isometric, we can know that
$$||h_0||^2=||e_0||^2=\frac{1}{1-|a|^2},$$
thus $|c_0|=(1-|a|^4)^{-1}$.
\end{proof}

\begin{claim}
For the function $h_0=Ce_0$ in the proof of Claim \ref{1}, we have
$$\langle h_0,\varphi^{3k}\rangle=c_0(1-|a|^4)\rho^k$$\
for $k=0,1,2...$
\end{claim}

\begin{proof}
\begin{align}\label{3.4}
\langle h_0,\varphi^{3k}\rangle=\sum_{j=0}^{k}c_j\overline{a}^{3k-3j}+\sum_{j=k+1}^{\infty}c_ja^{3j-3k}.
\end{align}
Comparing (\ref{3.4}) with (\ref{3.1}), we can get that
\begin{align}\label{3.5}
\langle h_0,\varphi^{3k}\rangle=(1-|a|^4)\sum_{j=0}^{k}c_j\overline{a}^{3k-3j}.
\end{align}
So (\ref{3.5}), along with (\ref{3.3}), shows that
\begin{align*}
\langle h_0,\varphi^{3k}\rangle&=(1-|a|^4)\frac{c_{k+1}}{\tilde\rho}
\\&=c_0(1-|a|^4)\rho^k.
\end{align*}
\end{proof}

\begin{claim}\label{4}
Under the assumption of Claim \ref{1}, we have
$$Ce_1=-c_0\frac{\overline{a}(1-|a|^6)}{a(1-|a|^4)}\cdot\frac{\varphi_a(1-\overline{a}^3\varphi_a^3)}{(1-\rho\varphi_a^3)^2}+\overline{a}h_0.$$
\end{claim}

\begin{proof}
Let $h_1=Ce_1$. Since $e_1-ae_0\in\Lambda_1^*$, we have $h_1-\overline{a}h_0\in\Lambda_1$. So we can assume that
$$h_1=\sum_{j=0}^\infty b_j\varphi_a^{3j+1}+\overline{a}h_0.$$

It is obvious that $e_1$ is orthogonal to $\Lambda_0^*$, so $h_1$ is orthogonal to $\Lambda_0$, which means that $\langle h_1,\varphi_a^{3k}\rangle=0$ for $k=0,1,2...$ So we have the following equations,
$$\sum_{j=0}^{\infty}b_ja^{3j+1}+c_0\overline{a}(1-|a|^4)=0,$$
and
$$\sum_{j=0}^{k-1}b_j\overline{a}^{3k-3j-1}+\sum_{j=k}^{\infty}b_ja^{3j+1-3k}+c_0\overline{a}(1-|a|^4)\rho^k=0$$
for $k=1,2,3...$

So
$$b_0a\frac{1-|a|^4}{1-|a|^6}+c_0\overline{a}=0,$$
and
$$\sum_{j=0}^{k-1}b_j\overline{a}^{3k-3j-1}+b_ka\frac{1-|a|^4}{1-|a|^6}+c_0\overline{a}\rho^k=0$$
for $k=1,2,3...$

Now let
$$\delta_j=-\frac{b_ja(1-|a|^4)}{c_0\overline{a}(1-|a|^6)},$$
then $\delta_0=1$, $\delta_1=\rho+\tilde\rho$, and
$$\delta_{k+1}=\rho\delta_k+\tilde\rho\rho^k$$
for $k=1,2,3...$ Hence we can get that
$$\delta_k=\rho^k+k\tilde\rho\rho^{k-1}$$
for $k=0,1,2...$

Thus we have
\begin{align*}
h_1-\overline{a}h_0&=\sum_{j=0}^\infty b_j\varphi_a^{3j+1}
\\&=-c_0\frac{\overline{a}(1-|a|^6)}{a(1-|a|^4)}\sum_{j=0}^\infty \delta_j\varphi_a^{3j+1}
\\&=-c_0\frac{\overline{a}(1-|a|^6)}{a(1-|a|^4)}\left(\sum_{j=0}^\infty\rho^j\varphi_a^{3j+1}+\sum_{j=0}^\infty j\tilde\rho\rho^{j-1}\varphi_a^{3j+1}\right)
\\&=-c_0\frac{\overline{a}(1-|a|^6)}{a(1-|a|^4)}\left(\frac{\varphi_a}{1-\rho\varphi_a^3}+\frac{\tilde\rho\varphi_a^4}{(1-\rho\varphi_a^3)^2}\right)
\\&=-c_0\frac{\overline{a}(1-|a|^6)}{a(1-|a|^4)}\cdot\frac{\varphi_a(1-\overline{a}^3\varphi_a^3)}{(1-\rho\varphi_a^3)^2}.
\end{align*}
\end{proof}

Now we can prove our final result as follows.

\begin{theorem}\label{main}
If $\varphi$ is an elliptic automorphism of order 3 with fixed point $a\in D\backslash \{0\}$, then $C_\varphi$ is not complex symmetric on $H^2(D)$.
\end{theorem}

\begin{proof}
Suppose that $C_\varphi$ is $C$-symmetric with respect to conjugation $C$. Then Claim \ref{1} to \ref{4} hold.

Let
$$f=\frac{1-|a|^6}{1-|a|^4}\cdot\frac{1-\overline{a}^3\varphi_a^3}{(1-\rho\varphi_a^3)^2},$$
Then $||f||=|c_0|^{-1}\cdot||h_1-\overline{a}h_0||=(1-|a|^4)||h_1-\overline{a}h_0||$.

A tedious calculation shows that $f=\beta_1g^2+\beta_2g+\beta_3$, where $g=\frac{\overline{\rho}-\varphi_a^3}{1-\rho\varphi_a^3}$ and
\begin{align*}
\beta_1&=\frac{(1-|a|^4)(1+|a|^2)^2}{1-|a|^6}(\rho^2-\overline{a}^3\rho)=\frac{\overline{a}^4}{a^2}(1+|a|^2);
\\\beta_2&=\frac{(1-|a|^4)(1+|a|^2)^2}{1-|a|^6}(-2\rho+\overline{a}^3+\overline{a}^3|\rho|^2)=\frac{\overline{a}^2}{a}(1+|a|^2)(2+|a|^2);
\\\beta_3&=\frac{(1-|a|^4)(1+|a|^2)^2}{1-|a|^6}(1-\overline{a}^3\overline{\rho})=(1+|a|^2)^2.
\end{align*}

So
\begin{align*}
||f||^2&=\langle\beta_1g^2+\beta_2g+\beta_3,\beta_1g^2+\beta_2g+\beta_3\rangle
\\&=|\beta_1|^2+|\beta_2|^2+|\beta_3|^2+2\Re\left(\beta_1\overline{\beta_2}g(0)+\beta_2\overline{\beta_3}g(0)+\beta_1\overline{\beta_3}g(0)^2\right)
\\&=(1+2|a|^2-2|a|^4-|a|^6)(1+|a|^2)^2.
\end{align*}

However, since $C$ is isometric,
\begin{align*}
||f||^2&=(1-|a|^4)^2||h_1-\overline{a}h_0||^2
\\&=(1-|a|^4)^2||e_1-\overline{a}e_0||^2
\\&=(1-|a|^4)^2(1+|a|^2)(1-|a|^2)
\\&=(1-|a|^4)(1+|a|^2)^2.
\end{align*}
So we have
\begin{align*}
(1+2|a|^2-2|a|^4-|a|^6)(1+|a|^2)^2&=(1-|a|^4)(1+|a|^2)^2
\\2|a|^2-|a|^4-|a|^6&=0
\\|a|^2+|a|^4&=2,
\end{align*}
which is impossible since $a\in D$.
\end{proof}

At last we can get our main result as a corollary. It gives a complete description of the automorphisms which can induce complex symmetric operators on $H^2(D)$.

\begin{corollary}\label{last}
Suppose $\varphi$ is an automorphism of $D$. Then $C_\varphi$ is complex symmetric on $H^2(D)$ if and only if $\varphi$ is either a rotation or an elliptic automorphism of order two.
\end{corollary}

\end{document}